\documentclass[12pt]{article}
\usepackage{amsfonts}
\usepackage{amsmath}
\usepackage{amssymb}
\usepackage{fullpage}
\usepackage{amsthm}

\newtheorem{theorem}{Theorem}
\newtheorem{proposition}[theorem]{Proposition}
\newtheorem{lemma}[theorem]{Lemma} 
\newtheorem{corollary}[theorem]{Corollary}

\newcommand{\calo}{{\mathcal O}}
\newcommand{\pdeg}{{p^\circ}}
\newcommand{\Sp}{{\hbox{\rm Sp}}}
\newcommand{\gl}{{\hbox{\rm GL}}}
\newcommand{\SL}{{\hbox{\rm SL}}}
\newcommand{\cc}{{\mathbb C}}
\newcommand{\ff}{{\mathbb F}}
\newcommand{\rr}{{\mathbb R}}
\newcommand{\zz}{{\mathbb Z}}
\newcommand{\tp}{\mathrm T}
\newcommand{\lcm}{\ensuremath{\operatorname{lcm}}}
\newcommand{\idm}{I_n}
\newcommand{\sce}{{\scriptstyle{E}}}
\newcommand{\scf}{{\scriptstyle{F}}}
\newcommand{\Ind}{\hbox{\rm Ind}}
\newcommand{\semi}{{\rtimes}}

\title{The Yagita invariant of symplectic groups of large 
rank}
\author{Cornelia M. Busch\thanks{The first author acknowledges support
    from ETH Z\"urich, which facilitated this work}\and Ian
  J. Leary\thanks{The second author would like to thank the Isaac
    Newton Institute for Mathematical Sciences for support and hospitality during the programme {\sl Non-positive Curvature, Group Actions and Cohomology}, 
when work on this paper was undertaken.  This work was supported by EPSRC grant
  no.\ EP/K032208/1 and by a grant from the Leverhulme Trust.}}

\date{\today}

\begin{document} 

\maketitle

\begin{abstract} 
Fix a prime $p$, and let $\calo$ denote a subring of $\cc$ that 
is either integrally closed or contains a primitive $p$th root 
of~1.  We determine the Yagita invariant at the prime $p$ for 
the symplectic group $\Sp(2n,\calo)$ for all $n\geq p-1$.  
\end{abstract}  

\section{Introduction} 
The Yagita invariant $\pdeg(G)$ of a discrete group $G$ is an
invariant that generalizes the period of the $p$-local Tate-Farrell
cohomology of $G$, in the following sense: it is a numerical 
invariant defined for any $G$ that is equal to the period when 
the $p$-local cohomology of $G$ is periodic.  
Yagita considered finite groups~\cite{yag}, and Thomas extended the
definition to groups of finite vcd~\cite{thomas}.  In~\cite{glt} 
the definition was extended to arbitrary groups and $\pdeg(G)$ was 
computed for $G=\gl(n,\calo)$ for $\calo$ any integrally closed subring of
$\cc$ and for sufficiently large $n$ (depending on $\calo$).  

In~\cite{busch}, one of us computed the Yagita invariant for
$\Sp(2(p+1),\zz)$.  Computations from~\cite{glt} were used to provide
an upper bound and computations with finite subgroups and with mapping
class groups were used to provide a lower bound~\cite{gmx}.  The
action of the mapping class group of a surface upon the first homology
of the surface gives a natural symplectic representation of the
mapping class group of a genus $p+1$ surface inside $\Sp(2(p+1),\zz)$.
In the current paper, we compute $\pdeg(\Sp(2n,\calo))$ for each
$n\geq p-1$ for each $\calo$ for which $\pdeg(\gl(n,\calo))$ was
computed in~\cite{glt}.  By using a greater range of finite subgroups
we avoid having to consider mapping class groups.  

Throughout the paper, we fix a prime $p$.  Before stating our main
result we recall the definitions of the symplectic group $\Sp(2n,R)$ over
a ring $R$, and of the Yagita invariant $\pdeg(G)$, which depends on
the prime $p$ as well as on the group $G$.  The group $\Sp(2n,R)$ is 
the collection of invertible $2n\times 2n$ matrices $M$ over $R$ such 
that 
\[M^\tp JM = J,\,\,\,\hbox{where}\,\,\,
J:= \begin{pmatrix}\phantom{-}0&\idm\\-\idm &0\end{pmatrix}.\] 
Here $M^\tp$ denotes the transpose of the matrix $M$, and as usual
$I_n$ denotes the $n\times n$ identity matrix.  Equivalently 
$M\in \Sp(2n,R)$ if $M$ defines an isometry of the antisymmetric 
bilinear form on $R^{2n}$ defined by $\langle x,y\rangle:=x^\tp Jy$.  
If $C$ is cyclic of order 
$p$, then the group cohomology ring $H^*(C;\zz)$ has the form
\[H^*(C;\zz)\cong \zz[x]/(px),\,\,\,\, x\in H^2(C;\zz).\] 
If $C$ is a cyclic subgroup of $G$ of order $p$, 
define $n(C)$ a positive integer or infinity to be 
the supremum of the integers $n$ such that the image of
$H^*(G;\zz)\rightarrow H^*(C;\zz)$ is contained in the subring 
$\zz[x^n]$.  Now define 
\[\pdeg(G):=\lcm\{2n(C)\,\,:\,\, C\leq G,\,\,\, |C|=p\}.\] 
It is easy to see that if $H\leq G$ then $\pdeg(H)$ divides
$\pdeg(G)$~\cite[Prop.~1]{glt}.  

\section{Results} 

In the following theorem statement and throughout the paper we let
$\zeta_p$ be a primitive $p$th root of 1 in $\cc$ and we let $\calo$
denote a subring of $\cc$ with $F\subseteq \cc$ as its field of 
fractions.  We assume that either $\zeta_p\in \calo$ or that
$\calo$ is integrally closed in $\cc$.  
We define $l:=|F[\zeta_p]:F|$, the degree of
$F[\zeta_p]$ as an extension of $F$.  For $t\in \rr$ with $t\geq 1$,
we define $\psi(t)$ to be the largest integer power of $p$ less than
or equal to $t$.

\begin{theorem} \label{thm:main}
With notation as above, for each $n\geq p-1$, the Yagita invariant 
$\pdeg(\Sp(2n,\calo))$ is equal to $2(p-1)\psi(2n/l)$ for $l$ even 
and equal to $2(p-1)\psi(n/l)$ for $l$ odd.  
\end{theorem} 

By the main result of~\cite{glt}, the above is equivalent to the 
statement that $\pdeg(\Sp(2n,\calo))=\pdeg(\gl(2n,\calo))$ when $l$ is
even and $\pdeg(\Sp(2n,\calo))=\pdeg(\gl(n,\calo))$ when $l$ is odd. 
By definition $\Sp(2n,\calo)$ is a subgroup of $\gl(2n,\calo)$ and 
there is an inclusion $\gl(n,\calo)\rightarrow \Sp(2n,\calo)$ defined
by 
\[A\mapsto \begin{pmatrix}A&0\\0 &(A^\tp)^{-1}\end{pmatrix},\] 
and so for any $n$, $\pdeg(\gl(n,\calo))$ divides $\pdeg(\Sp(2n,\calo))$, 
which in turn divides $\pdeg(\gl(2n,\calo))$.  

Before we start, we recall two standard facts concerning symplectic 
matrices that will be used in the proof of Corollary~\ref{cor:upper}: 
if $M$ is in the symplectic group 
then $\det(M)=1$ and $M$ is conjugate to the inverse of its 
transpose $(M^{-1})^\tp=(M^\tp)^{-1}$.  We shall use the notation 
$\ff_p^\times$ to denote the multiplicative group of units in the 
field $\ff_p$.  

\begin{proposition} \label{prop:main} 
Let $f(X)$ be a polynomial over the field $\ff_p$ and suppose that 
$0$ is not a root of $f$ but that $f$ factors as a product of linear 
polynomials over $\ff_p$.  If there is a polynomial $g$ and an integer
$n$ so that $f(X)=g(X^n)$, then $n$ has the form $n=mp^q$ for some $m$
dividing $p-1$ and some integer $q\geq 0$.  If $p$ is odd and for 
each $i\in\ff_p^\times$, the multiplicity of $i$ as a 
root of $f$ is equal to that of $-i$, then $m$ is even.  
\end{proposition} 

\begin{proof} 
The only part of this that is not contained in~\cite[Prop.~6]{glt} is 
the final statement.  Since $(1-iX)(1+iX)=1-i^2X^2$ is a polynomial in 
$X^2$, the final statement follows.  For the benefit of the reader, we 
sketch the rest of the proof.  If $n=mp^q$ where $p$ does not divide 
$m$, then $g(X^n)=g(X^m)^{p^q}$, so we may assume that $q=0$.  If
$g(Y)=0$ has roots $y_i$, then the roots of $g(X^m)=0$ are the roots
of $y_i-X^m=0$.  Since $p$ does not divide $m$, these polynomials have
no repeated roots; since their roots are assumed to lie in $\ff_p$ 
it is now easy to show that $m$ divides $p-1$.  
\end{proof} 

\begin{corollary}\label{cor:upper} 
With notation as in Theorem~\ref{thm:main}, let $G$ be a subgroup of
$\Sp(2n,F)$.  Then the Yagita invariant $\pdeg(G)$ divides the number
given for $\pdeg(\Sp(2n,\calo))$ in the statement of Theorem~\ref{thm:main}.  
\end{corollary} 

\begin{proof} 
As in~\cite[Cor.~7]{glt}, for each $C\leq G$ of order $p$, we use the
total Chern class to give an upper bound for the number $n(C)$
occurring in the definition of $\pdeg(G)$.  If $C$ is cyclic of order
$p$, then $C$ has $p$ distinct irreducible complex representations,
each 1-dimensional.  If we write $H^*(C;\zz)=\zz[x]/(px)$, then the
total Chern classes of these representations are $1+ix$ for each $i\in
\ff_p$, where $i=0$ corresponds to the trivial representation.  The
total Chern class of a direct sum of representations is the product of
the total Chern classes, and so when viewed as a polynomial in
$\ff_p[x]=H^*(C;\zz)\otimes \ff_p$, the total Chern class of any
faithful representation $\rho:C\rightarrow \gl(2n,\cc)$ is a
non-constant polynomial of degree at most $2n$ all of whose roots lie
in $\ff_p^\times$.  Now let $F$ be a subfield of $\cc$ with
$l=|F[\zeta_p]:F|$ as in the statement.  The group $C$ has $(p-1)/l$
non-trivial irreducible representations over $F$, each of
dimension~$l$, and the total Chern classes of these representations
have the form $1-ix^l$, where $i$ ranges over the $(p-1)/l$ distinct
$l$th roots of unity in $\ff_p$.  In particular, the total Chern class
of any representation $\rho:C\rightarrow \gl(2n,F)\leq \gl(2n,\cc)$ is
a polynomial in $x^l$ whose $x$-degree is at most $2n$.  If $\rho$ has
image contained in $\Sp(2n,\cc)$, then it factors as 
$\rho = \iota\circ \widetilde\rho$ 
with $\widetilde\rho:C\rightarrow\Sp(2n,\cc)$ and 
$\iota$ is the inclusion of $\Sp(2n,\cc)$ in $\gl(2n,\cc)$. 
In this case the matrix representing a
generator for $C$ is conjugate to the transpose of its own inverse; in
particular it follows that the multiplicities of the irreducible
complex representations of $C$ with total Chern classes $1+ix$ and
$1-ix$ must be equal for each $i$. Hence in this case, if $p$ is odd, 
the total Chern class of the representation 
$\rho=\iota\circ\widetilde\rho$ is a polynomial in $x^2$.  
If $p=2$ (which implies that $l=1$) then the total Chern class of any
representation $\rho:C\rightarrow \gl(2n,\cc)$ has the form $(1+x)^i$,
where $i$ is equal to the number of non-trivial irreducible summands.  
Since $\Sp(2n,\cc)\leq \SL(2n,\cc)$ it follows that for symplectic
representations $i$ must be even, and so, for $p=2$  
the total Chern class is a polynomial in $x^2$.

In summary, let $\widetilde\rho$ be a faithful representation of $C$ in
$\Sp(2n,F)$.  In the case when  
$l$ is odd, then the total Chern class of $\widetilde\rho$ is a 
non-constant polynomial 
$\tilde f(y)=f(x)$ in $y=x^{2l}$ such that $f(x)$ has degree at most
$2n$, $\tilde f(y)$ has degree at most $n/l$, and all roots of
$f,\tilde f$ lie in $\ff_p^\times$.  In the case when $l$ is even, the 
total Chern class of $\rho$ is a non-constant polynomial $\tilde f(y)=f(x)$
in $y=x^l$ such that $f(x)$ has degree at most $2n$, $\tilde f(y)$ has 
degree at most $2n/l$, and all roots of both lie in $\ff_p^\times$.  
By Proposition~\ref{prop:main}, it follows that each $n(C)$ is a
factor of the number given for $\pdeg(\Sp(2n,\calo))$, and hence the 
claim. 
\end{proof} 

\begin{lemma}\label{lem:induced}
Let $H\leq G$ with $|G:H|=m$, and let $\rho$ be a symplectic
representation of $H$ on $V=\calo^{2n}$.  The induced representation 
$\Ind_H^G(\rho)$ is a symplectic representation of $G$ on
$W:=\calo G\otimes_{\calo H}V\cong\calo^{2mn}$.  
\end{lemma} 

\begin{proof} 
Let $e_1,\ldots,e_n,f_1,\ldots,f_n$ be the standard basis for
$V=\calo^{2n}$, so that the bilinear form $\langle v,w\rangle:= 
v^\tp J w$ on $V$ is given by  
\[\langle e_i,e_j\rangle = 0 = \langle f_i,f_j\rangle,\,\,\,\, 
\langle e_i,f_j\rangle = -\langle f_i,e_j\rangle= \delta_{ij}.\] 
The representation $\rho$ is symplectic if and only if each $\rho(h)$ 
preserves this bilinear form.  

Let $t_1,\ldots,t_m$ be a left transversal to $H$ in $G$, so that  
$\calo G= \oplus_{i=1}^m t_i\calo H$ as right $\calo H$-modules.  
Define a bilinear form $\langle\,\, ,\,\,\rangle_W$ on $W$ by 
\[\left\langle \sum_{i=1}^m t_i\otimes v^i, \sum_{i=1}^m t_i\otimes
w^i\right\rangle_W := \sum_{i=1}^m \langle v^i,w^i\rangle.\] 
To see that this bilinear form is preserved by the $\calo G$-action 
on $W$, fix $g\in G$ and define a permutation $\pi$ of 
$\{1,\ldots,m\}$ and elements $h_1,\ldots, h_m\in H$ by the equations
$gt_i=t_{\pi(i)}h_i$.  Now for each $i,j$ with $1\leq i,j\leq m$  
\begin{align*}
\left\langle \Ind(\rho(g))t_i\otimes v, \Ind(\rho(g))t_j\otimes w
\right\rangle_W 
&= \left\langle t_{\pi(i)}\otimes \rho(h_i)v,
t_{\pi(j)}\otimes \rho(h_j)w \right\rangle_W \\
&=\delta_{\pi(i)\pi(j)} \langle
\rho(h_i)v,\rho(h_j)w\rangle\\
&= \delta_{ij}\langle \rho(h_i)v,\rho(h_i)w\rangle \\
&=\delta_{ij} \langle v,w\rangle\\ 
&= \langle t_i\otimes v,t_j\otimes w\rangle_W.
\end{align*}
To see that $\langle\,\, ,\,\, \rangle_W$ is symplectic, define basis
elements $\sce_1,\ldots,\sce_{mn},\scf_1,\ldots,\scf_{mn}$ 
for $W$ by the equations 
\[\sce_{n(i-1)+j}:=t_i\otimes
  e_j,\,\,\,\hbox{and}\,\,\,\scf_{n(i-1)+j}:= t_i\otimes f_j,\,\,\,\,\,
\hbox{for}\,\,1\leq i\leq m,\,\,1\leq j\leq n.\]  
It is easily checked that for $1\leq i,j\leq mn$ 
\[\langle \sce_i,\sce_j\rangle_W = 0 = \langle \scf_i,\scf_j\rangle_W,\,\,\,\, 
\langle \sce_i,\scf_j\rangle_W = -\langle \scf_i,\sce_j\rangle_W= \delta_{ij},\] 
and so with respect to this basis for $W$, the bilinear form $\langle \,\,
,\,\,\rangle_W$ is the standard symplectic form. 
\end{proof}

\begin{proposition} \label{prop:lower}
With notation as in Theorem~\ref{thm:main}, 
the Yagita invariant $\pdeg(\Sp(2n,\calo))$ is divisible by the number 
given in the statement of Theorem~\ref{thm:main}.  
\end{proposition}

\begin{proof}
To give lower bounds for $\pdeg(\Sp(2n,\calo))$ we use finite
subgroups.  Firstly, consider the semidirect product $H=C_p\semi
C_{p-1}$, where $C_{p-1}$ acts faithfully on $C_p$; equivalently this
is the group of affine transformations of the line over $\ff_p$.  It
is well known that the image of $H^*(G;\zz)$ inside $H^*(C_p;\zz)\cong
\zz[x]/(px)$ is the subring generated by $x^{p-1}$.  It follows that
$2(p-1)$ divides $\pdeg(G)$ for any $G$ containing $H$ as a subgroup.
The group $H$ has a faithful permutation action on $p$ points, and
hence a faithful representation in $\gl(p-1,\zz)$, where $\zz^{p-1}$
is identified with the kernel of the $H$-equivariant map
$\zz\{1,\ldots, p\}\rightarrow \zz$.  Since $\gl(p-1,\zz)$ embeds in
$\Sp(2(p-1),\zz)$ we deduce that $H$ embeds in $\Sp(2n,\calo)$ for
each $\calo$ and for each $n\geq p-1$.

To give a lower bound for the $p$-part of $\pdeg(\Sp(2n,\calo))$ we
use the extraspecial $p$-groups.  For $p$ odd, let $E(p,1)$ be the
non-abelian $p$-group of order $p^3$ and exponent $p$, and let
$E(2,1)$ be the dihedral group of order 8.  (Equivalently in each case
$E(p,1)$ is the Sylow $p$-subgroup of $\gl(3,\ff_p)$.)  For $m\geq 2$,
let $E(p,m)$ denote the central product of $m$ copies of $E(p,1)$, so
that $E(p,m)$ is one of the two extraspecial groups of order
$p^{2m+1}$.  Yagita showed that $\pdeg(E(p,m))=2p^m$ for each $m$ and
$p$~\cite{yag}.  The centre and commutator subgroup of $E(p,m)$ are
equal and have order $p$, and the abelianization of $E(p,m)$ is
isomorphic to $C_p^{2m}$.  The irreducible complex representations of
$E(p,m)$ are well understood: there are $p^{2m}$ distinct
1-dimensional irreducibles, each of which restricts to the centre as
the trivial representation, and there are $p-1$ faithful
representations of dimension $p^m$, each of which restricts to the
centre as the sum of $p^m$ copies of a single (non-trivial)
irreducible representation of $C_p$.  The group $G=E(p,m)$ contains a
subgroup $H$ isomorphic to $C_p^{m+1}$, and each of its faithful
$p^m$-dimensional representations can be obtained by inducing up a
1-dimensional representation $H\rightarrow C_p\rightarrow \gl(1,\cc)$.

According to B\"urgisser, $C_p$ embeds in $\Sp(2l,\calo)$ (resp.~in
$\Sp(l,\calo)$ when $l$ is even) provided that $\calo$ is integrally
closed in $\cc$~\cite{buerg}.  Here as usual, $l:=|F[\zeta_p],F|$ and
$F$ is the field of fractions of $\calo$.
 If instead $\zeta_p\in \calo$, then $l=1$ and clearly $C_p$ embeds 
in $\gl(1,\calo)$ and hence also in $\Sp(2,\calo)=\Sp(2l,\calo)$.  
Taking this embedding of $C_p$ and composing it with any homomorphism 
$H\rightarrow C_p$ we get a symplectic representation $\rho$ of $H$ 
on $\calo^{2l}$ for any $l$ (resp.~on $\calo^l$ for $l$ even).  For 
a suitable homomorphism we know that $\Ind_H^G(\rho)$ is a faithful
representation of $G$ on $\calo^{2lp^m}$ (resp.~on $\calo^{lp^m}$ for 
$l$ even) and by Lemma~\ref{lem:induced} we see that $\Ind_H^G(\rho)$ 
is symplectic.  Hence we see that $E(m,p)$ embeds as a subgroup of 
$\Sp(2lp^m,\calo)$ for any $l$ and as a subgroup of $\Sp(lp^m,\calo)$ in 
the case when $l$ is even.  Since $\pdeg(E(m,p))=2p^m$, this shows
that $2p^m$ divides $\pdeg(\Sp(2lp^m,\calo))$ always and that 
$2p^m$ divides $\pdeg(\Sp(lp^m,\calo))$ in the case when $l$ is even.  
\end{proof} 

Corollary~\ref{cor:upper} and Proposition~\ref{prop:lower} together
complete the proof of Theorem~\ref{thm:main}.  

We finish by pointing out that we have not computed $\pdeg(\Sp(2n,\calo))$ 
for general $\calo$ when $n<p-1$; to do this one would have to know
which metacyclic groups $C_p\semi C_k$ with $k$ coprime to $p$ admit 
low-dimensional symplectic representations.

\leftline{\bf Authors' addresses:}

\obeylines

\smallskip
{\tt cornelia.busch@math.ethz.ch} 

\smallskip
Department of Mathematics
ETH Z\"urich
R\"amistrasse 101 
8092 Z\"urich 
Switzerland

\smallskip
{\tt i.j.leary@soton.ac.uk}

\smallskip
School of Mathematical Sciences 
University of Southampton
Southampton
SO17 1BJ 
United Kingdom

\end{document}